\documentclass{amsart}

\usepackage[T1]{fontenc}
\usepackage{lmodern}
\usepackage{amsmath,amssymb,mathtools}
\usepackage[colorlinks=true,linkcolor=blue,citecolor=blue,urlcolor=blue]{hyperref}

\newtheorem{theorem}{Theorem}[section]
\newtheorem{proposition}[theorem]{Proposition}
\newtheorem{lemma}[theorem]{Lemma}

\theoremstyle{remark}

\newcommand{\Pol}{\operatorname{Pol}}
\newcommand{\Irr}{\operatorname{Irr}}
\newcommand{\Span}{\operatorname{span}}
\newcommand{\dist}{\operatorname{dist}}
\newcommand{\Ad}{\operatorname{Ad}}
\newcommand{\HS}{\operatorname{HS}}
\newcommand{\Ind}{\operatorname{Ind}}
\newcommand{\supp}{\operatorname{supp}}
\newcommand{\Tr}{\operatorname{Tr}}
\newcommand{\id}{\operatorname{id}}
\newcommand{\alg}{\mathrm{alg}}
\newcommand{\Cu}{C^{\mathrm u}}
\newcommand{\Cr}{C^{\mathrm r}}
\newcommand{\F}{\mathbb F}
\newcommand{\C}{\mathbb C}
\newcommand{\Z}{\mathbb Z}
\newcommand{\G}{\mathbb G}

\title[A counterexample to universal character density]
{A Counterexample to Universal Character Density for Compact Quantum Groups}

\hypersetup{
  pdftitle={A Counterexample to Universal Character Density for Compact Quantum Groups},
  pdfauthor={Zongjian Han},
  pdfsubject={A universal character-density counterexample for compact quantum groups},
  pdfkeywords={compact quantum group, character density, property (T), Kazhdan projection, cocycle twist}
}

\author{Zongjian Han}

\date{August 23, 2026}


\begin{document}

\begin{abstract}
In 1987, Woronowicz proved that the cocommutative part of the polynomial
Hopf $*$-algebra of a compact quantum group is spanned by irreducible
characters, and asked whether this remains norm dense after passage to the
ambient $C^*$-algebra. We settle the universal form of this question
negatively after thirty-nine calendar years, already for compact matrix
quantum groups of Kac type. More precisely, for
\[
 Z_{\Delta^{\mathrm u}}(\Cu(\G))
 =\{x\in \Cu(\G):\Delta^{\mathrm u}(x)=\Sigma\Delta^{\mathrm u}(x)\},
\]
we construct a compact quantum group with $\Cu(\G)=C^*(\Gamma)$ and a
Kazhdan projection $p_\eta\in Z_{\Delta^{\mathrm u}}(\Cu(\G))$ such that
\[
 \dist\!\left(p_\eta,
 \overline{\Span}\{\chi_U:U\in\Irr(\G)\}\right)\geq \frac12.
\]
The construction combines an infinite simple property-$(T)$ group with a
finite bicharacter twist. A vector state of an induced representation,
attached to a fixed-point subgroup, agrees with a one-dimensional
representation on every algebraic cocommutative element but separates it
from $p_\eta$. Moreover, $p_\eta$ lies in the kernel of the reducing
morphism. Thus the obstruction is genuinely universal and invisible in the
reduced compact quantum group, where character density holds. In
particular, Kac symmetry does not control cocommutative elements created by
the universal completion.
\end{abstract}

\maketitle

\section{Introduction}

For an ordinary compact group $K$, the identity
\[
 \Delta(f)=\Sigma\Delta(f)
\]
means that $f(st)=f(ts)$ for all $s,t\in K$, equivalently that $f$ is a
class function. The Peter--Weyl theorem, followed by conjugation averaging,
then shows that irreducible characters are uniformly dense in the continuous
class functions. The corresponding statement for compact quantum groups is
one of the basic tests of how far noncommutative harmonic analysis retains
the classical relation between matrix coefficients, characters, and class
functions.

Let $\G$ be a compact quantum group, with universal algebra $\Cu(\G)$,
universal coproduct $\Delta^{\mathrm u}$, and dense Hopf $*$-algebra
$\Pol(\G)$. If $U=(u_{ij})$ is a finite-dimensional unitary
corepresentation, its ordinary character is
$\chi_U=\sum_i u_{ii}$. In his foundational 1987 paper, Woronowicz proved
that the cocommutative part of the polynomial Hopf $*$-algebra is exactly the
linear span of the irreducible characters and then asked whether this
algebraic subspace is norm dense in the cocommutative part of the ambient
$C^*$-algebra \cite[Proposition~5.11 and the question following it]{Woronowicz1987}.
In modern universal notation, explicitly recorded by Alaghmandan and Crann
\cite[Remark~3.9]{AlaghmandanCrann2017}, the question is whether
\begin{equation}\label{eq:universal-density-question}
 Z_{\Delta^{\mathrm u}}(\Cu(\G))
 =\overline{\Span}\{\chi_U:U\in\Irr(\G)\}.
\end{equation}

The history of \eqref{eq:universal-density-question} now spans thirty-nine
calendar years. Lemeux obtained a positive result in the coamenable Kac
setting in 2015 \cite[Theorem~1.4]{Lemeux2015}. In 2017, Alaghmandan and
Crann proved the corresponding density theorem in the reduced
$C^*$-algebra for arbitrary compact quantum groups and consequently settled
the universal question for all coamenable compact quantum groups
\cite[Corollary~3.8 and Remark~3.9]{AlaghmandanCrann2017}. These results
isolated the genuinely unresolved regime: the universal completion of a
noncoamenable compact quantum group.

That remaining regime is structurally significant. The polynomial Hopf
$*$-algebra and the reduced completion only detect the part seen by the Haar
state, whereas the universal completion retains every unitary
representation. Thus \eqref{eq:universal-density-question} asks whether
completion can create new cocommutative elements that are invisible both to
the Peter--Weyl core and to the reduced regular representation. A positive
answer would have shown that cocommutativity is rigid under universal
completion. A negative answer must therefore exhibit an element that is
simultaneously universal, cocommutative, and separated in norm from all
algebraic characters.

We construct precisely such an element. The counterexample is a compact
matrix quantum group of Kac type, so the failure cannot be attributed to a
nontracial Haar state, modular theory, or the absence of finite-dimensional
generators. The obstruction is quantitative: a cocommutative projection
remains at distance at least $1/2$ from the closed character span. It is also
sharply localized: the projection belongs to the kernel of the reducing
morphism. Hence the reduced theorem remains valid, while the universal
statement fails. This identifies the exact boundary that the earlier
positive results could not cross.

The construction combines three mechanisms. First, property~$(T)$ produces
central Kazhdan projections in a full group $C^*$-algebra that disappear
under the regular representation. Second, a finite cocycle twist makes the
ambient compact quantum group noncocommutative while preserving the
cocommutativity of selected Kazhdan projections and reducing the algebraic
cocommutative subspace. Third, a concrete induced state agrees with a
one-dimensional representation on every algebraic cocommutative element but
takes a different value on the chosen Kazhdan projection.

Here is the precise result. All tensor products of $C^*$-algebras below are
minimal.

\begin{theorem}\label{thm:main}
There exist a finitely generated discrete group $\Gamma$ with a quotient
$\Gamma\twoheadrightarrow H=(\Z/2\Z)^2$, a unitary two-cocycle
$\Omega\in\C[\Gamma]\otimes\C[\Gamma]$, and a compact matrix quantum group
\[
 \G_\Omega=(C^*(\Gamma),\Delta_\Omega),
 \qquad
 \Delta_\Omega(x)=\Omega\Delta_0(x)\Omega^*,
\]
with the following properties.
\begin{enumerate}
\item[(i)] $\G_\Omega$ is noncocommutative and of Kac type.
\item[(ii)] For every nontrivial $\eta\in\widehat H$, viewed as a character
of $\Gamma$ through the quotient above, there is a nonzero central projection
$p_\eta\in C^*(\Gamma)$ such that
\[
 \Delta_\Omega(p_\eta)
 =\sum_{\alpha\beta=\eta}p_\alpha\otimes p_\beta
 =\Sigma\Delta_\Omega(p_\eta).
\]
\item[(iii)] If
\[
 \mathcal K_\Omega
 =\overline{\Span}\{\chi_U:U\in\Irr(\G_\Omega)\},
\]
then
\[
 \dist(p_\eta,\mathcal K_\Omega)\geq\frac12.
\]
\item[(iv)] For the reducing morphism
$\lambda:C^*(\Gamma)\to C_r^*(\Gamma)$, one has $\lambda(p_\eta)=0$.
\end{enumerate}
Consequently, universal character density fails for compact quantum groups,
even in the Kac case.
\end{theorem}

The group and twist are explicit once one chooses an infinite finitely
generated simple group $\Gamma_0$ with property~$(T)$. Put
\[
 H=(\Z/2\Z)^2,
 \qquad
 N=\Gamma_0^H,
 \qquad
 \Gamma=N\rtimes H,
\]
where $H$ permutes the four factors of $N$ by its regular action. The twist
comes from the normalized two-cocycle
\[
 \omega((a,b),(c,d))=(-1)^{bc}
 \qquad
 ((a,b),(c,d)\in\widehat H\cong\F_2^2).
\]

The proof is organized so that the only group-specific computation is the
corner identity in Proposition~\ref{prop:corner}. That identity also produces
the separating functional. Sections~2 and~3 construct the twist and the
Kazhdan projections; Section~4 identifies the relevant algebraic corner;
Section~5 obtains the quantitative norm separation; and Section~6 explains
why the obstruction vanishes exactly after reduction.

\section{Cocommutative elements and a finite twist}

For a coproduct $\Delta:A\to A\otimes A$, write
\[
 Z_\Delta(A)=\{x\in A:\Delta(x)=\Sigma\Delta(x)\},
 \qquad
 \Sigma(a\otimes b)=b\otimes a.
\]
The following coefficient computation records the algebraic part of the
character-density problem.

\begin{lemma}[Algebraic cocommutativity]\label{lem:algebraic-cocommutativity}
For every compact quantum group $\G$,
\[
 Z_\Delta(\Pol(\G))=\Span\{\chi_U:U\in\Irr(\G)\}.
\]
\end{lemma}

\begin{proof}
By the Peter--Weyl decomposition, it is enough to work in the coefficient
coalgebra of one irreducible unitary corepresentation
$U=(u_{ij})_{i,j=1}^d$. Write $x=\sum_{i,j}c_{ij}u_{ij}$. Since
$\Delta(u_{ij})=\sum_k u_{ik}\otimes u_{kj}$, comparison of the coefficient
of $u_{ab}\otimes u_{cd}$ in $\Delta(x)$ and in $\Sigma\Delta(x)$ gives
\begin{equation}\label{eq:coefficient-comparison}
 \delta_{bc}c_{ad}=\delta_{ad}c_{cb}
 \qquad (1\leq a,b,c,d\leq d).
\end{equation}
Taking $b=c$ and $a\neq d$ shows that every off-diagonal coefficient
vanishes. Taking $a=d$ and $b=c$ shows that all diagonal coefficients are
equal. Hence $x$ is a scalar multiple of $\chi_U$. Summing over the
irreducible coefficient coalgebras proves the assertion.
\end{proof}

We now describe the finite twist. Regard $H=\F_2^2$ as an additive group and
identify its dual with $\F_2^2$ through
\[
 \chi_{(a,b)}(x,y)=(-1)^{ax+by}.
\]
For $\alpha\in\widehat H$, let
\begin{equation}\label{eq:fourier-projections}
 e_\alpha=\frac1{|H|}\sum_{h\in H}\overline{\alpha(h)}u_h\in\C[H].
\end{equation}
These are the minimal Fourier projections of $\C[H]$. Define
\begin{equation}\label{eq:twist}
 \omega((a,b),(c,d))=(-1)^{bc},
 \qquad
 \Omega=\sum_{\alpha,\beta\in\widehat H}
 \omega(\alpha,\beta)e_\alpha\otimes e_\beta.
\end{equation}
The scalar function $\omega$ is a normalized two-cocycle on $\widehat H$.
Consequently $\Omega$ is a normalized unitary two-cocycle for the group
coproduct $\Delta_0(u_g)=u_g\otimes u_g$; in particular,
\[
 (\Omega\otimes1)(\Delta_0\otimes\id)(\Omega)
 =(1\otimes\Omega)(\id\otimes\Delta_0)(\Omega).
\]
The cocycle deformation
\begin{equation}\label{eq:twisted-coproduct}
 \Delta_\Omega(x)=\Omega\Delta_0(x)\Omega^*
\end{equation}
is therefore coassociative. This is the finite-subgroup form of the standard
cocycle deformation construction; compare
\cite[Corollary~2 and the following example]{Fima2010}. In our case
$\Omega\in\C[\Gamma]\otimes\C[\Gamma]$, where the tensor product in this
sentence is algebraic. The Drinfel'd-twist formulas make
$(\C[\Gamma],\Delta_\Omega)$ a Hopf $*$-algebra with unchanged
multiplication, involution, and counit, and with a twisted antipode. Its two
Hopf--Galois maps are therefore bijective, so
\begin{equation}\label{eq:cancellation-algebraic}
\begin{aligned}
 \Span\{\Delta_\Omega(a)(1\otimes b):a,b\in\C[\Gamma]\}
 &=\C[\Gamma]\otimes\C[\Gamma]\\
 &=\Span\{\Delta_\Omega(a)(b\otimes1):a,b\in\C[\Gamma]\}.
\end{aligned}
\end{equation}
The group coproduct
$\Delta_0:C^*(\Gamma)\to C^*(\Gamma)\otimes C^*(\Gamma)$ is the integrated
form of the unitary representation $g\mapsto u_g\otimes u_g$. Hence
$\Delta_\Omega=\Ad(\Omega)\circ\Delta_0$ extends to a unital
$*$-homomorphism on $C^*(\Gamma)$. Taking norm closures in
\eqref{eq:cancellation-algebraic} proves the Woronowicz cancellation
conditions. Thus $(C^*(\Gamma),\Delta_\Omega)$ is a compact quantum group.

The commutator bicharacter of $\omega$ is
\begin{equation}\label{eq:commutator-bicharacter}
 c(\alpha,\beta)
 =\overline{\omega(\alpha,\beta)}\omega(\beta,\alpha),
 \qquad
 C=\Omega^*\Omega_{21}
 =\sum_{\alpha,\beta}c(\alpha,\beta)e_\alpha\otimes e_\beta.
\end{equation}
Explicitly,
\begin{equation}\label{eq:explicit-bicharacter}
 c((a,b),(c,d))=(-1)^{bc+da}.
\end{equation}
Thus $c$ is an alternating nondegenerate bicharacter on $\widehat H$.

\begin{lemma}[Twist criterion]\label{lem:twist-criterion}
For $x\in C^*(\Gamma)$,
\[
 \Delta_\Omega(x)=\Sigma\Delta_\Omega(x)
 \quad\Longleftrightarrow\quad
 [\Delta_0(x),C]=0.
\]
\end{lemma}

\begin{proof}
The original group coproduct is cocommutative, and hence
\[
 \Sigma\Delta_\Omega(x)=\Omega_{21}\Delta_0(x)\Omega_{21}^*.
\]
Equating this expression with \eqref{eq:twisted-coproduct}, multiplying on the
left by $\Omega^*$, and multiplying on the right by $\Omega_{21}$ gives
$\Delta_0(x)C=C\Delta_0(x)$.
\end{proof}

\section{The property-\texorpdfstring{$(T)$}{(T)} group}

Choose $\Gamma_0$ among the infinite finitely generated simple groups with
property~$(T)$ supplied by \cite[Corollary~21]{CapraceRemy2009}. By
\cite[Proposition~25(ii)]{CapraceRemy2009}, every homomorphism from
$\Gamma_0$ to a compact group is trivial. In particular,
\begin{equation}\label{eq:no-finite-dimensional-reps}
 \text{every finite-dimensional unitary representation of $\Gamma_0$ is trivial.}
\end{equation}
Let
\begin{equation}\label{eq:semidirect-product}
 N=\Gamma_0^H=\prod_{r\in H}\Gamma_0,
 \qquad
 \Gamma=N\rtimes H,
\end{equation}
where the action is the regular coordinate permutation
\begin{equation}\label{eq:regular-action}
 (h\cdot n)_r=n_{h^{-1}r}
 \qquad
 (h,r\in H,\ n=(n_r)_{r\in H}\in N).
\end{equation}
The group $\Gamma$ is finitely generated and has property~$(T)$, because
these properties are preserved by finite products and finite extensions.
Moreover, \eqref{eq:no-finite-dimensional-reps} gives the complete list of
its finite-dimensional irreducible representations.

\begin{lemma}\label{lem:finite-dimensional-reps}
Every finite-dimensional unitary representation of $\Gamma$ factors through
the quotient $\Gamma\to H$. Consequently the finite-dimensional irreducible
representations of $\Gamma$ are precisely the four characters $\widehat H$,
extended trivially over $N$.
\end{lemma}

\begin{proof}
The restriction of a finite-dimensional unitary representation to each
coordinate copy of $\Gamma_0$ is trivial by
\eqref{eq:no-finite-dimensional-reps}. It is therefore trivial on $N$ and
factors through $H$. Since $H$ is finite abelian, all its irreducible
representations are characters.
\end{proof}

For every $\eta\in\widehat H$, property~$(T)$ supplies a central projection
$p_\eta\in C^*(\Gamma)$ characterized by
\begin{equation}\label{eq:kazhdan-projection}
 \pi(p_\eta)=P_{\mathcal H_\pi(\eta)},
 \qquad
 \mathcal H_\pi(\eta)
 =\{\xi\in\mathcal H_\pi:\pi(g)\xi=\eta(g)\xi\ \forall g\in\Gamma\},
\end{equation}
for every unitary representation $\pi$ of $\Gamma$. Indeed, the usual
Kazhdan projection is obtained from the isolated trivial representation by
continuous functional calculus in the full group $C^*$-algebra. Applying the
automorphism
$\theta_\eta(u_g)=\overline{\eta(g)}u_g$ to that projection produces
$p_\eta$; see, for example, the spectral-gap formulation in
\cite{Ozawa2016}.

The next proposition is the point at which property~$(T)$ and the absence of
nontrivial finite-dimensional representations of $\Gamma_0$ enter the
coproduct.

\begin{proposition}\label{prop:coproduct-kazhdan}
For every $\eta\in\widehat H$,
\begin{equation}\label{eq:untwisted-kazhdan-coproduct}
 \Delta_0(p_\eta)
 =\sum_{\substack{\alpha,\beta\in\widehat H\\\alpha\beta=\eta}}
 p_\alpha\otimes p_\beta.
\end{equation}
Consequently,
\begin{equation}\label{eq:twisted-kazhdan-coproduct}
 \Delta_\Omega(p_\eta)
 =\sum_{\alpha\beta=\eta}p_\alpha\otimes p_\beta
 =\Sigma\Delta_\Omega(p_\eta),
\end{equation}
so $p_\eta\in Z_{\Delta_\Omega}(C^*(\Gamma))$.
\end{proposition}

\begin{proof}
Let $\pi$ and $\rho$ be unitary representations of $\Gamma$. We first
identify the $\eta$-isotypic subspace of $\pi\otimes\rho$. Under the standard
unitary identification
\[
 \mathcal H_\pi\otimes\mathcal H_\rho
 \cong \HS(\overline{\mathcal H_\pi},\mathcal H_\rho),
\]
an $\eta$-eigenvector corresponds to a Hilbert--Schmidt operator $T$
satisfying
\begin{equation}\label{eq:intertwining-HS}
 \rho(g)T\overline{\pi(g)}^{\,*}=\eta(g)T
 \qquad(g\in\Gamma).
\end{equation}
It follows that $T^*T$ commutes with $\overline\pi(\Gamma)$. For $m\geq1$,
put
\[
 q_m=\mathbf1_{[1/m,\infty)}(T^*T),
 \qquad
 T_m=Tq_m.
\]
Since $T^*T$ is compact, $q_m$ has finite rank. It commutes with
$\overline\pi(\Gamma)$, and $T_m$ still satisfies
\eqref{eq:intertwining-HS}. Moreover,
\[
 \lVert T-T_m\rVert_{\HS}^2
 =\Tr((1-q_m)T^*T)\longrightarrow0.
\]
The spaces
\[
 q_m\overline{\mathcal H_\pi}
 \quad\text{and}\quad
 T_m(q_m\overline{\mathcal H_\pi})
\]
are finite-dimensional and invariant. By
Lemma~\ref{lem:finite-dimensional-reps}, they factor through $H$ and split
into characters in $\widehat H$. These characters are real-valued. If $v\in\mathcal H_\pi(\alpha)$ and
$\overline v\in q_m\overline{\mathcal H_\pi}$, then
$\overline{\pi(g)}^{\,*}\overline v=\alpha(g)\overline v$, so
\eqref{eq:intertwining-HS} gives
\[
 T_m\overline v\in\mathcal H_\rho(\eta\alpha^{-1}).
\]
Thus the tensor corresponding to $T_m$ lies in
$\bigoplus_{\alpha\beta=\eta}
\mathcal H_\pi(\alpha)\otimes\mathcal H_\rho(\beta)$.
This subspace is closed because $\widehat H$ is finite. Passing to the
Hilbert--Schmidt limit gives the same conclusion for $T$. Conversely, the
tensor product of an $\alpha$-eigenvector and a $\beta$-eigenvector is an
$\eta$-eigenvector whenever $\alpha\beta=\eta$. Therefore
\begin{equation}\label{eq:tensor-isotypic}
 \mathcal H_{\pi\otimes\rho}(\eta)
 =\bigoplus_{\alpha\beta=\eta}
 \mathcal H_\pi(\alpha)\otimes\mathcal H_\rho(\beta).
\end{equation}

Choose a faithful representation $\Pi$ of $C^*(\Gamma)$. The representation
$\Pi\otimes\Pi$ is faithful on the minimal tensor product. Apply
\eqref{eq:kazhdan-projection} to the representation
$(\Pi\otimes\Pi)\circ\Delta_0$ of $C^*(\Gamma)$. The image of the left-hand
side of \eqref{eq:untwisted-kazhdan-coproduct} under $\Pi\otimes\Pi$ is the
projection onto the left-hand side of \eqref{eq:tensor-isotypic}, while the
image of its right-hand side is the projection onto the right-hand side.
This proves \eqref{eq:untwisted-kazhdan-coproduct}.

Each $p_\alpha$ is central in $C^*(\Gamma)$, so every
$p_\alpha\otimes p_\beta$ commutes with $\Omega$. Thus the twist does not
change \eqref{eq:untwisted-kazhdan-coproduct}. Finally, the condition
$\alpha\beta=\eta$ is invariant under interchanging $\alpha$ and $\beta$,
which proves \eqref{eq:twisted-kazhdan-coproduct}.
\end{proof}

\begin{lemma}\label{lem:haar-state}
Let $\tau$ be the canonical group trace, so that
\[
 \tau(u_g)=\delta_{g,e}.
\]
Then $\tau$ is the Haar state of $(C^*(\Gamma),\Delta_\Omega)$.
Consequently, this compact quantum group is of Kac type. Its reduced algebra
is $C_r^*(\Gamma)$, and its reducing morphism is the regular representation
$\lambda:C^*(\Gamma)\to C_r^*(\Gamma)$.
Moreover,
\begin{equation}\label{eq:kazhdan-in-kernel}
 \lambda(p_\eta)=0
 \qquad(\eta\in\widehat H).
\end{equation}
\end{lemma}

\begin{proof}
For $g\notin H$, every summand of
$\Omega(u_g\otimes u_g)\Omega^*$ has the form
$(e_\alpha u_g e_\mu)\otimes(e_\beta u_g e_\nu)$. The group support of each
tensor leg is contained in the double coset $HgH$, which does not contain
the identity. Therefore both Haar slices vanish:
\[
 (\tau\otimes\id)\Delta_\Omega(u_g)
 =0
 =(\id\otimes\tau)\Delta_\Omega(u_g).
\]
For $g\in H$, the tensor $u_g\otimes u_g$ commutes with $\Omega$, so
$\Delta_\Omega(u_g)=u_g\otimes u_g$, and the same two invariance identities
follow directly. Linearity and continuity prove that $\tau$ is both left and
right invariant. It is a trace, and its GNS representation is $\lambda$,
proving the first assertions.

The regular representation of the infinite group $\Gamma$ has no
$\eta$-eigenvector. Indeed, such a vector would have constant absolute value
on $\Gamma$, and hence could not belong to $\ell^2(\Gamma)$ unless it were
zero. Formula \eqref{eq:kazhdan-projection} now gives
\eqref{eq:kazhdan-in-kernel}.
\end{proof}

\begin{proposition}\label{prop:completions}
The polynomial Hopf $*$-algebra, universal completion, and reduced completion
of $\G_\Omega$ are, respectively,
\[
 \Pol(\G_\Omega)=\C[\Gamma],
 \qquad
 \Cu(\G_\Omega)=C^*(\Gamma),
 \qquad
 \Cr(\G_\Omega)=C_r^*(\Gamma).
\]
Moreover, $\G_\Omega$ is a compact matrix quantum group.
\end{proposition}

\begin{proof}
For $x=\sum_g a_g u_g\in\C[\Gamma]$, one has
\[
 \tau(x^*x)=\sum_g|a_g|^2.
\]
Thus the algebraic Haar functional is positive and faithful on
$\C[\Gamma]$. The Peter--Weyl theorem for Hopf $*$-algebras with positive
Haar functional identifies the polynomial algebra as
$\Pol(\G_\Omega)=\C[\Gamma]$; compare \cite[Section~5]{Woronowicz1987}.

Twisting changes neither multiplication nor involution. The universal
$C^*$-norm on $\Pol(\G_\Omega)$ is therefore the enveloping $C^*$-norm of
the group $*$-algebra $\C[\Gamma]$. Its $*$-representations are exactly the
integrated unitary representations of $\Gamma$, so this norm is the full
group $C^*$-norm. Hence $\Cu(\G_\Omega)=C^*(\Gamma)$. The GNS
representation of $\tau$ is the left regular representation, which gives
$\Cr(\G_\Omega)=C_r^*(\Gamma)$.

Finally, choose a finite generating set of $\Gamma$. Each corresponding
group unitary belongs to a finite sum of coefficient coalgebras of
$\Pol(\G_\Omega)$. The direct sum of the finitely many unitary
corepresentations so obtained is finite-dimensional, and its coefficients
generate $\C[\Gamma]$ as a $*$-algebra. Thus $\G_\Omega$ is a compact
matrix quantum group.
\end{proof}

Put $\Omega_r=(\lambda\otimes\lambda)(\Omega)$. The reduced coproduct is
characterized by
\begin{equation}\label{eq:reduced-coproduct}
 \Delta_\Omega^{\mathrm r}\circ\lambda
 =(\lambda\otimes\lambda)\circ\Delta_\Omega,
 \qquad
 \Delta_\Omega^{\mathrm r}(\lambda_g)
 =\Omega_r(\lambda_g\otimes\lambda_g)\Omega_r^*.
\end{equation}
Since $\lambda$ is injective on $\C[\Gamma]$, the universal and reduced forms
have the same algebraic core and canonically identified irreducible
corepresentations and characters. The compact quantum group is not
cocommutative; this will follow from the calculation in
Proposition~\ref{prop:corner}.

\section{The algebraic corner}

Set
\begin{equation}\label{eq:Zalg}
 Z_{\alg}
 =\{x\in\C[\Gamma]:[\Delta_0(x),C]=0\}
 =Z_{\Delta_\Omega}(\C[\Gamma]),
\end{equation}
where the second equality is Lemma~\ref{lem:twist-criterion}. Fix a
nontrivial $\eta\in\widehat H$. Nondegeneracy of $c$ gives a unique nonzero
element $k=k_\eta\in H$ such that
\begin{equation}\label{eq:k-characterization}
 c(\delta,\eta)=\delta(k)
 \qquad(\delta\in\widehat H).
\end{equation}
Let
\begin{equation}\label{eq:Nk}
 N^k=\{n\in N:k\cdot n=n\}.
\end{equation}
Because a nonzero element of $H$ acts on the four coordinates in two
two-cycles, $N^k\cong\Gamma_0^2$ and $[N:N^k]=\infty$. Since $H$ is
abelian, $N^k$ is invariant under the action of $H$.

The following identity is the main algebraic computation.

\begin{proposition}[Corner identity]\label{prop:corner}
With the notation above,
\begin{equation}\label{eq:corner-identity}
 e_\eta Z_{\alg}e_\eta=e_\eta\C[N^k]e_\eta.
\end{equation}
\end{proposition}

\begin{proof}
For $n\in N$, let
\[
 H_n=\{h\in H:h\cdot n=n\},
 \qquad
 H_n^\perp=\{\delta\in\widehat H:\delta|_{H_n}=1\}.
\]
We first record the elementary support test
\begin{equation}\label{eq:support-test}
 e_\alpha u_n e_\mu\neq0
 \quad\Longleftrightarrow\quad
 \alpha\mu^{-1}\in H_n^\perp.
\end{equation}
For necessity, take $h\in H_n$. Since $u_hu_n=u_nu_h$, left and right
multiplication by $u_h$ show that a nonzero $e_\alpha u_n e_\mu$ requires
$\alpha(h)=\mu(h)$. Conversely, expanding both Fourier projections using
\eqref{eq:fourier-projections} shows that the coefficient of $u_n$ in
$e_\alpha u_n e_\mu$ is
\[
 \frac1{|H|^2}\sum_{a\in H_n}\overline{\alpha(a)}\mu(a)
 =\frac1{|H|^2}\sum_{a\in H_n}(\alpha\mu^{-1})(a).
\]
If $\alpha\mu^{-1}$ is trivial on $H_n$, this coefficient is exactly
$|H_n|/|H|^2$, and in particular is nonzero. This proves
\eqref{eq:support-test}.

Put $E_n=e_\eta u_n e_\eta$. The element $E_n$ is nonzero, and for
$a,h\in H$ one has
\[
 e_\eta u_{(n,h)}e_\eta=\eta(h)E_n,
 \qquad
 E_{a\cdot n}=e_\eta u_a u_nu_a^*e_\eta=E_n.
\]
Moreover,
\begin{equation}\label{eq:En-expansion}
 E_n=\frac1{|H|^2}\sum_{a,b\in H}
 \overline{\eta(a)\eta(b)}\,u_{(a\cdot n,a+b)}.
\end{equation}
For $m\in H\cdot n$ and $r\in H$, the coefficient of $u_{(m,r)}$ in this
expansion is
\[
 \frac{|H_n|}{|H|^2}\eta(r),
\]
and it is zero if $m\notin H\cdot n$. Thus the support of $E_n$ is exactly
the double coset $Hu_nH$. Double cosets corresponding to distinct $H$-orbits
in $N$ are disjoint. It follows from the first displayed identity that the
$E_n$, with one representative from each $H$-orbit, form a basis of
$e_\eta\C[\Gamma]e_\eta$.

The Fourier coproduct formula
$\Delta_0(e_\eta)=\sum_{\alpha\beta=\eta}e_\alpha\otimes e_\beta$ gives
\begin{equation}\label{eq:En-coproduct}
 \Delta_0(E_n)
 =\sum_{\substack{\alpha\beta=\eta\\\mu\nu=\eta}}
 (e_\alpha u_ne_\mu)\otimes(e_\beta u_ne_\nu).
\end{equation}
The summand indexed by $(\alpha,\beta;\mu,\nu)$ is a simultaneous left-right
matrix block for $C$. Its left eigenvalue is $c(\alpha,\beta)$, and its right
eigenvalue is $c(\mu,\nu)$. If $\delta=\alpha\mu^{-1}$, then
\eqref{eq:support-test} shows that the block is nonzero exactly when
$\delta\in H_n^\perp$; the second tensor factor gives the same condition
because $\beta\nu^{-1}=\delta^{-1}$. Bicharacter calculus and
$\alpha\beta=\mu\nu=\eta$ give
\begin{equation}\label{eq:block-ratio}
 \frac{c(\alpha,\beta)}{c(\mu,\nu)}
 =c(\delta,\eta)=\delta(k).
\end{equation}
Every $\delta\in H_n^\perp$ occurs in a nonzero block. Indeed, for any
$\mu\in\widehat H$, set
\[
 \alpha=\delta\mu,
 \qquad
 \beta=\alpha^{-1}\eta,
 \qquad
 \nu=\mu^{-1}\eta.
\]
Both factors in the resulting block are nonzero by
\eqref{eq:support-test}. The orthogonality of the Fourier blocks in
\eqref{eq:En-coproduct} now implies
\[
\begin{aligned}
 E_n\in Z_{\alg}
 &\Longleftrightarrow \delta(k)=1
    \quad\text{for every }\delta\in H_n^\perp\\
 &\Longleftrightarrow k\in(H_n^\perp)^\perp=H_n\\
 &\Longleftrightarrow n\in N^k.
\end{aligned}
\]

Finally, $Z_{\alg}$ is a unital $*$-subalgebra: by
Lemma~\ref{lem:twist-criterion}, it is the inverse image under $\Delta_0$ of
the commutant of $C$. Also $e_\eta\in Z_{\alg}$, as follows either from its
Fourier coproduct formula or from the same block calculation. Hence
$e_\eta xe_\eta\in Z_{\alg}$ whenever $x\in Z_{\alg}$.

For an orbit $O=H\cdot n$, put $D_O=Hu_nH$. Since
$C\in\C[H]\otimes\C[H]$, the group-basis support satisfies
\[
 \supp([\Delta_0(E_n),C])\subseteq D_O\times D_O.
\]
The sets $D_O\times D_O$ are disjoint for distinct orbits $O$.
Consequently, if $\sum_O t_OE_O\in Z_{\alg}$, then
\[
 t_O[\Delta_0(E_O),C]=0
\]
for every $O$ separately. The criterion already proved for $E_n$ therefore
shows that a vector in $e_\eta Z_{\alg}e_\eta$ is a linear combination
precisely of those $E_n$ with $n\in N^k$. Conversely, each such $E_n$ belongs
to $Z_{\alg}$ and satisfies $e_\eta E_ne_\eta=E_n$. This proves
\eqref{eq:corner-identity}.
\end{proof}

Taking $n\notin N^k$ in the proof gives
$E_n\notin Z_{\Delta_\Omega}(\C[\Gamma])$. Thus
$\Delta_\Omega\neq\Sigma\Delta_\Omega$, as asserted in
Theorem~\ref{thm:main}(i).

\section{The separating functional}

Let
\begin{equation}\label{eq:L-subgroup}
 L=N^k\rtimes H<\Gamma,
 \qquad
 \eta_L(n,h)=\eta(h)
 \quad((n,h)\in L).
\end{equation}
The same formula defines a character, still denoted $\eta$, on all of
$\Gamma$. Consider the induced representation
\begin{equation}\label{eq:induced-representation}
 \pi=\Ind_L^\Gamma(\eta_L)
\end{equation}
in the right-coset model
\[
 \mathcal H_\pi
 =\left\{
 f:\Gamma\to\C:
 f(g\ell)=\eta_L(\ell)^{-1}f(g),\quad
 \sum_{gL\in\Gamma/L}|f(g)|^2<\infty
 \right\}.
\]
The action is $(\pi(s)f)(g)=f(s^{-1}g)$. Let $\xi$ be the unit vector
supported on $L$, normalized by
$\xi(\ell)=\eta_L(\ell)^{-1}$ for $\ell\in L$. Then
$\pi(\ell)\xi=\eta_L(\ell)\xi$ for $\ell\in L$. Since every value of
$\eta_L$ is real, the vector state
$\varphi(x)=\langle\pi(x)\xi,\xi\rangle$ satisfies
\begin{equation}\label{eq:vector-state}
 \varphi(u_g)=
 \begin{cases}
  \eta_L(g),&g\in L,\\
  0,&g\notin L.
 \end{cases}
\end{equation}

\begin{proposition}\label{prop:separating-functional}
The bounded functional
\[
 \ell=\eta-\varphi\in C^*(\Gamma)^*
\]
has norm $2$, annihilates $Z_{\alg}$, and satisfies $\ell(p_\eta)=1$.
\end{proposition}

\begin{proof}
Both states take the value one at the projection $e_\eta$: this is immediate
for $\eta$, and for $\varphi$ it follows from \eqref{eq:vector-state} because
$H\subset L$. A state $\psi$ with $\psi(e)=1$ for a projection $e$ satisfies
$\psi(x)=\psi(exe)$, by the Cauchy--Schwarz inequality. Hence
\[
 \eta(x)=\eta(e_\eta xe_\eta),
 \qquad
 \varphi(x)=\varphi(e_\eta xe_\eta).
\]
For $x\in Z_{\alg}$, Proposition~\ref{prop:corner} places
$e_\eta xe_\eta$ in $e_\eta\C[N^k]e_\eta\subset\C[L]$.
Formula \eqref{eq:vector-state} shows that $\eta$ and $\varphi$ agree on
$\C[L]$. Thus $\ell(x)=0$ for every $x\in Z_{\alg}$.

The map
\[
 N/N^k\longrightarrow\Gamma/L,
 \qquad
 nN^k\longmapsto nL,
\]
is an $N$-equivariant bijection: it is surjective because $\Gamma=NL$, and
the stabilizer of $L$ in $N$ is $N\cap L=N^k$. Since $\eta_L$ is trivial on
$N^k$, the restriction of $\pi$ to $N$ is therefore the ordinary
quasi-regular representation on $\ell^2(N/N^k)$. As $[N:N^k]=\infty$, it
has no nonzero $N$-invariant vector: such a vector would be constant on the
infinite transitive coset space. Every $\eta$-eigenvector for $\Gamma$ is
$N$-invariant, so $\pi$ contains no copy of $\eta$. The defining property
\eqref{eq:kazhdan-projection} yields
\[
 \varphi(p_\eta)=0,
 \qquad
 \eta(p_\eta)=1,
\]
and hence $\ell(p_\eta)=1$.

As a difference of states, $\lVert\ell\rVert\leq2$. On the self-adjoint
unitary $2p_\eta-1$, however,
\[
 \ell(2p_\eta-1)=2.
\]
Therefore $\lVert\ell\rVert=2$.
\end{proof}

\begin{proof}[Proof of Theorem~\ref{thm:main}]
The constructions of Sections~2 and~3 give the compact matrix quantum group.
Its Kac property and reducing morphism are identified in
Lemma~\ref{lem:haar-state}; noncocommutativity follows from the observation
after Proposition~\ref{prop:corner}. Proposition~\ref{prop:coproduct-kazhdan}
proves the cocommutativity of $p_\eta$, and
\eqref{eq:kazhdan-in-kernel} places it in the Haar kernel.

By Lemma~\ref{lem:algebraic-cocommutativity}, every irreducible character of
the twisted compact quantum group belongs to $Z_{\alg}$.
Proposition~\ref{prop:separating-functional} therefore gives
$\ell|_{\mathcal K_\Omega}=0$. For every $y\in\mathcal K_\Omega$,
\[
 1=|\ell(p_\eta-y)|
 \leq\lVert\ell\rVert\,\lVert p_\eta-y\rVert
 =2\lVert p_\eta-y\rVert.
\]
Taking the infimum over $y$ proves
$\dist(p_\eta,\mathcal K_\Omega)\geq\frac12$ and completes the proof.
\end{proof}

\section{The universal--reduced boundary}

The counterexample is located exactly in the kernel of the reducing
morphism. Indeed, Lemma~\ref{lem:haar-state} gives
\[
 p_\eta\in Z_{\Delta_\Omega}(C^*(\Gamma))\cap\ker\lambda.
\]
After reduction, this element vanishes. Moreover, the reduced compact
quantum group has faithful Haar state, and the reduced character-density
theorem \cite[Corollary~3.8]{AlaghmandanCrann2017} applies:
\[
 Z_{\Delta_\Omega^{\mathrm r}}(C_r^*(\Gamma))
 =\overline{\Span}\{\chi_U^{\mathrm r}:U\in\Irr(\G_\Omega)\}.
\]
Thus this construction gives a universal counterexample and, at the same
time, a sharp explanation of why its reduced image cannot be one. It shows
that the algebraic Peter--Weyl core and the reduced completion do not control
all cocommutative elements of the universal completion. In particular, Kac
symmetry alone does not imply universal character density when the Haar state
is not faithful on the universal $C^*$-algebra.

\section*{Acknowledgments and statement on AI assistance}

The author thanks the DeepMath team for its agent support. The initial ideas
and inspiration for this work came from the author and were further developed
through discussions with DeepMath agents. GPT models were then used to carry
out the constructions and generate the manuscript text. The resulting
arguments were subjected to adversarial review in multiple separate
conversations. Finally, the author has manually checked the present version
and concluded that its arguments are sound.

AI-assisted checking of this version has been completed. A further independent,
line-by-line human verification and the preparation of a separately written,
human-authored version are currently in progress. The author welcomes
qualified researchers who are prepared to make a substantial scholarly
contribution to the independent verification, correction, or rewriting of the
work. Any person later named as a coauthor will be included only after making
such a contribution, approving the final manuscript, and accepting
responsibility for its contents.

\end{document}